\documentclass{amsart}

\usepackage{amsmath,amssymb,amsthm}
\usepackage{a4wide}
\usepackage{hyperref}

\usepackage{graphicx}
\usepackage{xypic}
\entrymodifiers={+!!<0pt,\fontdimen22\textfont2>}
\usepackage[all]{xy}

\newtheoremstyle{myremark} 
    {7pt}                    
    {7pt}                    
    {}  	                 
    {}                           
    {\bf}       	         
    {.}                          
    {.5em}                       
    {}  

\theoremstyle{plain}
\newtheorem{lemma}{Lemma}[section]

\newtheorem{fact}[lemma]{Fact}
\newtheorem{definition}[lemma]{Definition}
\newtheorem{corollary}[lemma]{Corollary}
\newtheorem{proposition}[lemma]{Proposition}

\theoremstyle{myremark}
\newtheorem{remark}[lemma]{Remark}
\newtheorem{example}[lemma]{Example}

\newcommand{\rat}{\mathbb{Q}}

\newcommand{\inter}{}

\newcommand{\ind}{I}

\newcommand{\htpyequiv}{\simeq}

\newcommand{\incl}{\hookrightarrow}
\renewcommand{\subset}{\subseteq}
\newcommand{\susp}{\Sigma}
\newcommand{\lk}{\mathrm{lk}}
\newcommand{\st}{\mathrm{st}}
\newcommand{\hexdim}{\mathbb{H}}
\newcommand{\redhom}{\widetilde{H}}
\newcommand{\redbetti}{\widetilde{\beta}}

\newcommand{\tri}{\boldsymbol{\Delta}}

\newcommand{\Michal}[1]{}

\begin{document}
\title{Special cycles in independence complexes and superfrustration in some lattices}

\author[Micha{\l} Adamaszek]{Micha{\l} Adamaszek}
\address{Mathematics Institute and DIMAP,
      \newline University of Warwick, Coventry, CV4 7AL, UK}
\email{aszek@mimuw.edu.pl}
\thanks{Research supported by the Centre for Discrete
        Mathematics and its Applications (DIMAP), EPSRC award EP/D063191/1.}


\begin{abstract}
We prove that the independence complexes of some grids have exponential Betti numbers. This corresponds to the number of ground states in the hard-core model in statistical physics with fermions in the vertices of the grid.
\end{abstract}
\maketitle

\section{Introduction}
\label{sect:intro}
The purpose of this paper is to investigate some topological questions arising from the study of supersymmetric lattice models in statistical physics. Suppose one has a finite graph $L$, which in applications is usually a periodic lattice with some boundary conditions. Square, triangular or hexagonal grids are the most notable examples. The vertices of the graph can be occupied by particles, such as fermions, which satisfy the \emph{hard-core} restriction: two adjacent vertices cannot be occupied simultaneously. A configuration of particles which satisfies this assumption is an \emph{independent set} in the graph $L$.

Associated with $L$ there is a simplicial complex called the \emph{independence complex} of $L$ and denoted $\ind(L)$. Its vertices are the vertices of $L$ and its faces are the independent sets in $L$. It is a standard object studied in combinatorial algebraic topology. There is a close connection between the simplicial and topological invariants of $\ind(L)$ and certain characteristics of the corresponding lattice model which are of interest to physicists. It is beyond the scope of this paper to discuss this relationship in detail; we refer to \cite{HuSch2} and we limit ourselves to presenting just the most basic dictionary:
\begin{center}
\begin{tabular}{l|l}
the partition function of $L$ & the $f$-polynomial of $\ind(L)$,\\
the Witten index of $L$ & minus the reduced Euler characteristic $-\widetilde{\chi}(\ind(L))$,\\
the number of zero energy ground states & the dimension of $\redhom_*(\ind(L);\rat)$.
\end{tabular}
\end{center}
There has been some very successful work calculating the Witten index, homology groups or indeed the complete homotopy type of the independence complex for various lattices, eg. \cite{BLN,Fen2,HHFS,HuSch1,HuOthers,JonCylinder,JonRhombus,JonDiag,Thapper}. In this paper we focus on the large-scale picture. Computer simulations of van Eerten \cite{Eer} indicate that for some types of lattices, as their size increases, the number of ground states grows exponentially with the number of vertices, that is
$$\dim \redhom_*(\ind(L))\sim a^{v(L)}$$
for some constant $a$ depending on the type of the lattice, where $v(L)$ denotes the number of vertices in a graph $L$. This situation is called \emph{superfrustration} and has interesting physical implications, see \cite{HuSch1}. Engstr\"om \cite{EngWitten} developed a general method of computing \emph{upper} bounds for the constant $a$. For the lattices of \cite{Eer} it gives bounds very close to the values predicted in \cite{Eer}. 

This paper has two main parts. In the first one we present a method which can be used to construct exponentially many linearly independent homology classes in $\ind(L)$ for graphs $L$ of certain type. That proves superfrustration of certain lattices and we give examples based on modifications of the triangular lattice. In the second part we prove a generalization of the main result of \cite{EngWitten}, which can sometimes give better upper bounds.

Both methods work particularly nicely with one type of lattice studied in \cite{Eer,EngWitten}: the \emph{hexagonal dimer}, also known as the \emph{Kagome lattice} (see Fig.\ref{fig:hexdimer}). Under suitable divisibility conditions on the height and width we will prove that a graph $\hexdim$ of that type satisfies:
$$1.02^{v(\hexdim)}\approx (2^{1/36})^{v(\hexdim)} \leq \dim \redhom_*(\ind(\hexdim)) \leq (14^{1/36}\cdot2^{1/6})^{v(\hexdim)}\approx 1.21^{v(\hexdim)}.$$
We will prove the lower bound in Section \ref{sect:hexdimer} and the upper bound in Section \ref{sect:upperbounds}. The previous upper bound of \cite{EngWitten} was $2^{1/3}\approx 1.26$ and the experimental approximation by \cite{Eer} is $1.25\pm0.1$.

Our technique for lower bounds produces slightly more than just homology classes: we obtain a large wedge of spheres that splits off. For instance, for a suitable lattice $\hexdim$ of Kagome type, this reads as a homotopy equivalence
$$\ind(\hexdim)\htpyequiv \Big(\bigvee^{(2^{1/36})^{v(\hexdim)}} S^{2v(\hexdim)/9-1}\Big) \vee X$$
for some space $X$. This type of result is proved in Section \ref{sect:homotopical}.

\begin{remark}
Estimations as above can be compared against the absolute upper bound: for any graph $G$ we have
$$\dim \redhom_*(\ind(G)) \leq (2^{2/5})^{v(G)}\approx 1.32^{v(G)}.$$
This follows from the results of \cite{KozHulls}; for another short proof see \cite{Ada}.
\end{remark}

\subsection*{Acknowledgement.} The author thanks Alexander Engstr\"om, Liza Huijse and Kareljan Schoutens for very helpful suggestions.

\subsection{Notation.}
For a graph $G$ we denote by $V(G)$ the set of vertices and by $v(G)$ its cardinality. For any vertex $v$ we write $N[v]$ for the \emph{closed neighbourhood} of $v$, that is the set consisting of $v$ and all its adjacent vertices. For any set $W\subset V(G)$ we define $N[W]=\cup_{v\in W} N[v]$. 

For any simplicial complex $K$ and a subset $U\subset V(K)$ of the vertices $K[U]$ denotes the induced subcomplex of $K$ with vertex set $U$. The same notation is used for graphs. If $H$ is an induced subgraph of $G$ then $\ind(H)$ is an induced subcomplex of $\ind(G)$. By $|K|$ we denote the number of faces in $K$, including the empty one. The join $K\ast L$ of two complexes $K$ and $L$ is the simplicial complex with faces of the form $\sigma\sqcup\tau$ for $\sigma\in K$ and $\tau\in L$. For any two graphs $G$ and $H$, if $G\sqcup H$ is their disjoint union, we have
$$\ind(G\sqcup H)=\ind(G)\ast\ind(H).$$
The symbol $S^k$ denotes the topological sphere of dimension $k$. The suspension $\susp\,K$ is the join $K\ast S^0$. If $e$ is the graph consisting of a single edge then $\ind(e)=S^0$.

The reduced homology and cohomology groups of $K$, denoted $\redhom_*(K)$, $\redhom^*(K)$, are the homology groups of the augmented chain, resp. cochain complex of $K$. Throughout the paper we always use rational coefficients and omit them from notation. We have $\redhom_i(\susp\,K)=\redhom_{i-1}(K)$. There is a standard bilinear pairing, denoted $\langle\cdot,\cdot\rangle$:
$$\langle\cdot,\cdot\rangle:\redhom^i(K)\otimes \redhom_i(K)\to\rat$$
given by evaluating cochains on chains.
The $i$-th reduced Betti number is $\redbetti_i(K)=\dim\redhom_i(K)$ and the \emph{total Betti number} of $K$ is $$\redbetti(K)=\sum_i\redbetti_i(K).$$
It satisfies $\redbetti(\susp\,K)=\redbetti(K)$.

Note the empty simplicial complex $\emptyset$ which has no vertices and a unique face $\emptyset$. It satisfies $\susp\,\emptyset=S^0$, so it is good to think of it as $S^{-1}$. It has a single reduced homology group $\redhom_{-1}(\emptyset)=\rat$ and in particular $\redbetti(\emptyset)=1$. Of course for every non-empty space $K$ we have $\redhom_i(K)=0$ for $i<0$.

We do not distinguish between a simplicial complex and its geometric realization. The symbol $\htpyequiv$ means homotopy equivalence. The reference for other notions of combinatorial algebraic topology is \cite{Book}.

\section{Cycles defined by matchings}
A \emph{matching} of size $k$ is the disjoint union of $k$ edges. An \emph{induced matching} in a graph $G$ is a matching which is an induced subgraph of $G$. Explicitly, it is a set of $k$ edges of $G$ such that any two vertices from distinct edges are non-adjacent in $G$.

\begin{definition}
Suppose $M$ is an induced matching in a graph $G$. A subset $\sigma$ of the vertices of $M$ will be called a \emph{transversal} if it contains exactly one vertex from each edge of $M$. A transversal is \emph{dominating} if it is a dominating set in $G$ i.e. if $N[\sigma]=V(G)$.
\end{definition}

Every transversal of an induced matching is an independent set in $G$. It is dominating if and only if it is a maximal independent set.

If $M$ is a disjoint union of $k$ edges then $\ind(M)=S^0\ast\cdots\ast S^0=S^{k-1}$ is the boundary of the cross-polytope. If $M$ is an induced matching in $G$ then we obtain an embedding $S^{k-1}=\ind(M)\incl\ind(G)$. We will denote by $\alpha_M\in \redhom_{k-1}(\ind(G))$ the image of the fundamental class of $S^{k-1}$ under this embedding (there is a choice of orientations involved, but it does not matter). Explicitly, if $M=\{(v_1,w_1),(v_2,w_2),\ldots,(v_k,w_k)\}$ then $\alpha_M$ is represented by the cycle
\begin{equation}
([v_1]-[w_1])\wedge([v_2]-[w_2])\wedge\cdots\wedge([v_k]-[w_k]).
\end{equation}

If $\sigma$ is an independent set of cardinality $k$ in $G$ then we denote by $\sigma^\vee\in C^{k-1}(\ind(G))$ the cochain which associates $\pm 1$ to the two orientations of $\sigma$ and $0$ to all other simplices (again, this depends, up to sign, on the choice of orientation for $\sigma$). If $\sigma$ is a maximal independent set (hence a maximal face in $\ind(G)$) then this cochain is in fact a cocycle, hence it determines an element of $\redhom^{k-1}(\ind(G))$ which we continue to denote $\sigma^\vee$.

\begin{lemma}
If $(M,\sigma)$ is an induced matching with a dominating transversal then $\alpha_M$ and $\sigma^\vee$ are nonzero classes in $\redhom_*(\ind(G))$ and $\redhom^*(\ind(G))$, respectively. Moreover
$$\langle\sigma^\vee, \alpha_M\rangle=\pm 1.$$
\end{lemma}
\begin{proof}
The last statement holds because $\sigma^\vee$ evaluates to $\pm 1$ on exactly one of the simplices in the chain representation of $\alpha_M$. It immediately implies that both elements are nonzero.
\end{proof}

This method of constructing homology classes $\alpha_M$ has been used before, also in the context of grids \cite{JonGrids} where they are called \emph{cross-cycles}.

\begin{example}
Consider the cycle $C_6$ with vertex set $V(C_6)=\{1,2,3,4,5,6\}$ and with two matchings with dominating transversals:
\begin{eqnarray*}
M_1=\{(1,2),(4,5)\}, & \sigma_1=\{1,4\},\\
M_2=\{(2,3),(5,6)\}, & \sigma_2=\{2,5\}.
\end{eqnarray*}
Assuming that the above order defines positively oriented simplices we get:
$$\langle\sigma_i^\vee, \alpha_{M_i}\rangle=1 \textrm{ for } i=1,2,\quad \langle\sigma_1^\vee, \alpha_{M_2}\rangle=0, \quad \langle\sigma_2^\vee, \alpha_{M_1}\rangle=1.$$
It follows that $\alpha_{M_1}$ and $\alpha_{M_2}$ are linearly independent elements of $\redhom_1(\ind(C_6))$, hence they are a basis, as $\ind(C_6)\htpyequiv S^1\vee S^1$ (see \cite{Koz}). It means that if $M_3=\{(3,4),(6,1)\}$ is the third matching of this kind in $C_6$ then we must have $\alpha_{M_3}=a_1\alpha_{M_1}+a_2\alpha_{M_2}$ for some $a_1,a_2$. Applying $\sigma_1^\vee$ and $\sigma_2^\vee$ to this equation we get $-1=a_1$ and $0=a_1+a_2$ hence $\alpha_{M_3}=-\alpha_{M_1}+\alpha_{M_2}$.
\end{example}

\begin{remark}
In general, if $\alpha_1,\ldots,\alpha_k\in \redhom_i(X)$ and $\gamma_1,\ldots,\gamma_l\in \redhom^i(X)$ then the $i$-th Betti number $\redbetti_i(X)$ is at least as big as the rank of the $l\times k$ matrix with entries $\langle\gamma_s, \alpha_t\rangle$ for ${1\leq s\leq l}$, ${1\leq t\leq k}$.
\end{remark}

\section{Hexagonal dimer and related grids}
\label{sect:hexdimer}
The \emph{hexagonal dimer} or the \emph{Kagome lattice} is the lattice obtained from the triangular lattice by erasing every other line in each direction in the way shown in Fig.\ref{fig:hexdimer}. It is invariant under the translations by $(2,0)$ and $(0,\sqrt{3})$. Let $\hexdim_{n,m}$ denote the quotient of that lattice by the action of the translation group generated by the vectors $n\cdot(2,0)$ and $m\cdot(0,\sqrt{3})$. It has $v(\hexdim_{n,m})=3nm$ vertices and if one additionally assumes that $6|n$ and $4|m$ then it can be tiled with large hexagons in the way shown in Fig.\ref{fig:hexdimer}. 

In fact there is no harm forgetting about $n$ and $m$. Let $\hexdim$ be any quotient of the hexagonal dimer lattice (i.e. its finite portion with cyclic boundary conditions) with the property that it can be covered in this way by the large hexagons. Then the number of those hexagons is always $v(\hexdim)/36$ and we have the next result.

\begin{center}
\begin{figure}
\includegraphics[scale=0.8]{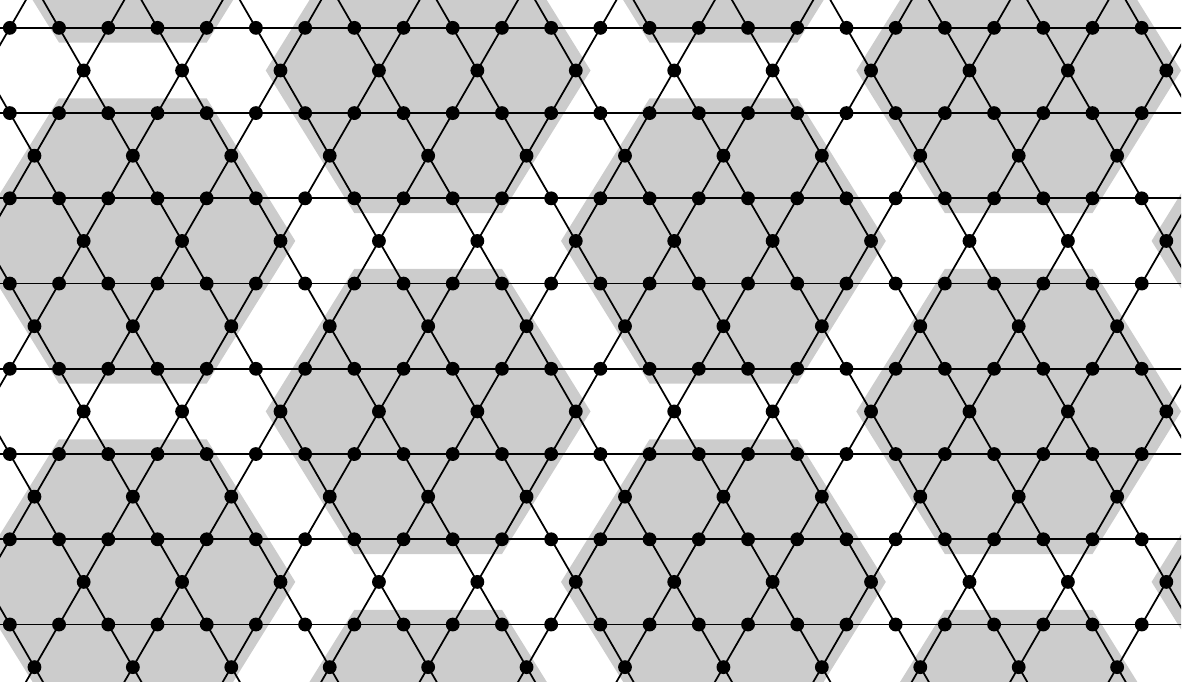}
\vskip0.3cm	
\begin{tabular}{cc}\includegraphics[scale=0.8]{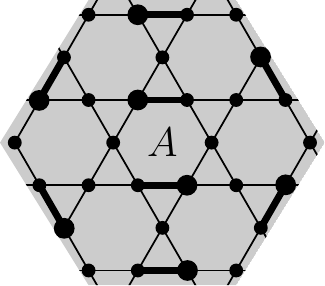} & \includegraphics[scale=0.8]{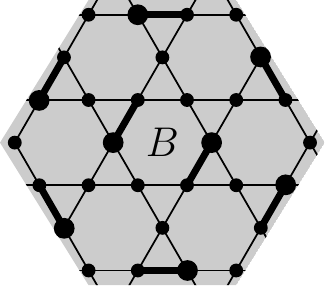}
\end{tabular}
\caption{The hexagonal dimer lattice, its tiling with large hexagons and two types of tiles.}
\label{fig:hexdimer}
\end{figure}
\end{center}

\begin{proposition}
\label{prop:lowerboundh}
For a hexagonal dimer lattice $\hexdim$ which admits a tiling as in Fig.\ref{fig:hexdimer} we have
$$\redbetti(\ind(\hexdim))\geq (2^{1/36})^{v(\hexdim)}.$$
\end{proposition}
\begin{proof}
Let $k=v(\hexdim)/36$ be the number of large hexagons in the tiling of $\hexdim$. We are going to construct $2^k$ linearly independent elements in $\redhom_{8k-1}(\ind(\hexdim))$. That proves the claim since then $\redbetti(\ind(\hexdim))\geq 2^k=2^{v(\hexdim)/36}$.

Consider the tiles A and B of Fig.\ref{fig:hexdimer}, where they are shown with an induced matching (thick edges) and its dominating transversal (thick vertices). Note that for any placement of $A$ and $B$ in place of the large hexagons in the grid we obtain a valid induced matching with a dominating transversal in $\hexdim$. Indeed, there are no edges directly between the tiles, so the matching is induced and each vertex located between the tiles is adjacent to one of the transversal vertices on the outer cycle of a tile. We therefore have $2^k$ homology (and cohomology) classes in $\ind(\hexdim)$. 

To prove linear independence we need some notation. Suppose that the big hexagons in the grid are labeled $1,\ldots,k$ in some order. For any sequence $s=(s_1,\ldots,s_k)$ of letters $A$, $B$ let $M(s)$ and $\sigma(s)$ denote the matching and transversal obtained in the above construction by placing the tile of type $s_i$ in the $i$-th spot for $i=1,\ldots,k$. Then for any two sequences $s$ and $t$ we have
\begin{equation*}
\langle\sigma(t)^\vee, \alpha_{M(s)}\rangle=\left\{\begin{array}{ll}0 & \textrm{ if } t_i=B \textrm{ and } s_i=A \textrm{ for some } i\\ \pm 1 & \textrm{ otherwise}\end{array}\right.
\end{equation*}
Consider a $2^k\times 2^k$ matrix with rows and columns indexed by sequences in $\{A,B\}^k$ in lexicographical order, where the entry in column $t$ and row $s$ is $\langle\sigma(t)^\vee, \alpha_{M(s)}\rangle$. By the last observation that matrix has $\pm 1$ on the diagonal and $0$ above the diagonal (where $t>_{LEX}s$), so it is of full rank. Therefore all $\alpha_{M(s)}$ are linearly independent.
\end{proof}

Similar results can be obtained for some other lattices derived from the triangular lattice.

\begin{proposition}
\label{prop:other}
For $d=3,4$ let $\tri_d$ be any finite quotient of the lattice obtained by removing every $d$-th line in each direction from the triangular lattice, such that the new lattice admits a tiling as in Fig.\ref{fig:otherlattices}. Then we have
$$\redbetti(\ind(\tri_3))\geq (2^{1/8})^{v(\tri_3)}, \qquad \redbetti(\ind(\tri_4))\geq (2^{1/45})^{v(\tri_4)}.$$
\end{proposition}
\begin{proof}
The proof is similar: $\tri_3$ has $k=v(\tri_3)/8$ tiles and $\tri_4$ has $k=v(\tri_4)/45$ of them. In each case we construct $2^k$ induced matchings with dominating transversals using the tiles $A$, $B$ of appropriate type (Fig.\ref{fig:otherlattices}) and prove linear independence of the resulting homology classes as before.
\end{proof}

\begin{center}
\begin{figure}
\begin{tabular}{cc}
\includegraphics[scale=0.9]{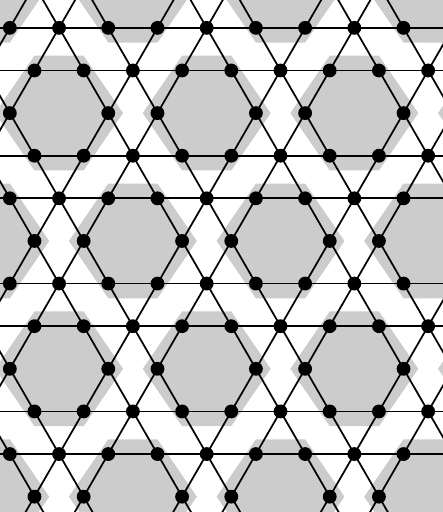} & \includegraphics[scale=0.7]{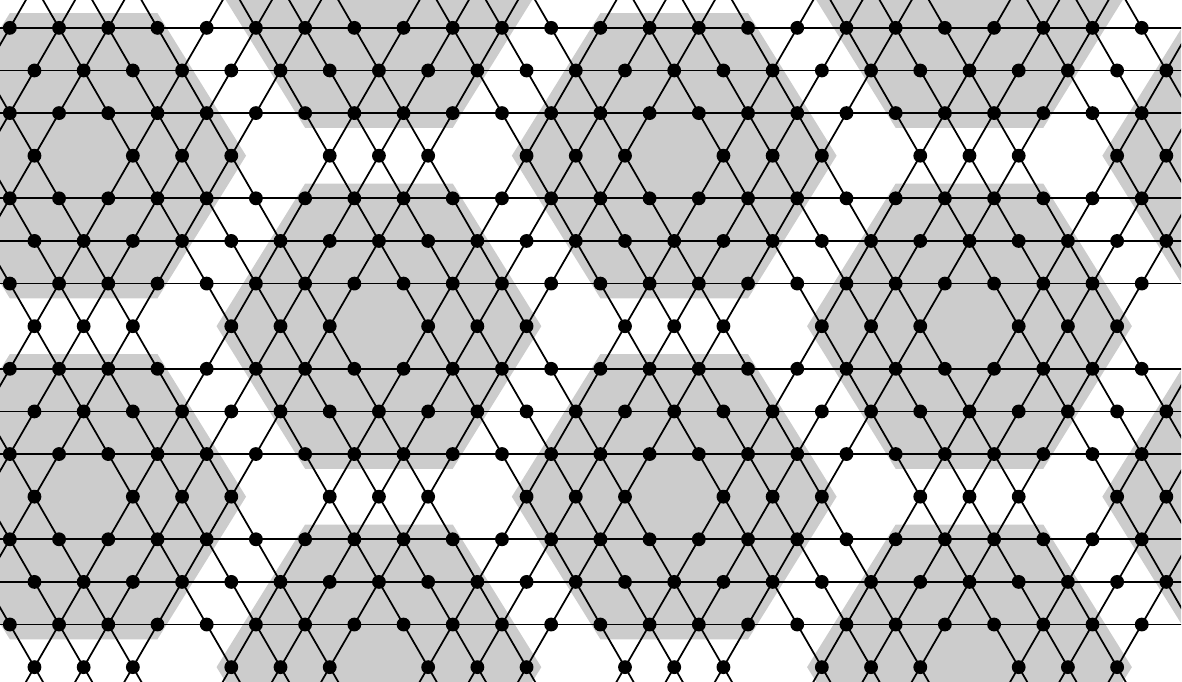} \\
\includegraphics[scale=1.1]{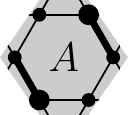} \includegraphics[scale=1.1]{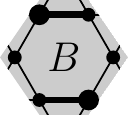} & \includegraphics[scale=0.7]{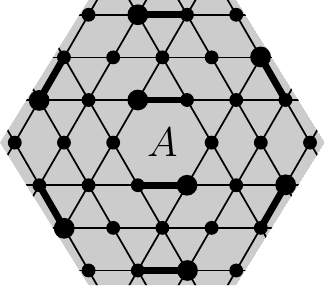} \includegraphics[scale=0.7]{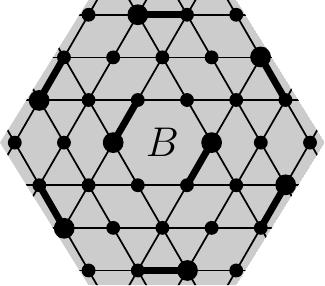}
\end{tabular}
\caption{Two other lattices derived from the triangular lattice and their sets of tiles.}
\label{fig:otherlattices}
\end{figure}
\end{center}

\section{Homotopical splittings}
\label{sect:homotopical}

We will now show for completeness that induced matchings with dominating transversals in $G$ correspond to sphere wedge summands in the homotopy type of $\ind(G)$. This is stronger than just saying that they define a nonzero homology class. The results of this section are not needed in the previous calculations, where analyzing the cohomology-homology pairing is sufficient (and easier).


\begin{lemma}
\label{lem:retract}
Let $(M,\sigma)$ be an induced matching of size $k$ with a dominating transversal in $G$. Denote by $j:\ind(M)\incl\ind(G)$ the embedding induced by the inclusion $i:M\incl G$. Then there is a homotopy equivalence
$$\ind(G)\htpyequiv \ind(M) \vee C(j)$$
where $C(j)$ is the homotopy cofibre of $j$.
\end{lemma}
\begin{proof}
Since $\sigma$ is a maximal face of $\ind(G)$ we can remove it and form the simplicial complex $\ind(G)\setminus \inter \sigma$ whose geometric realization is obtained by removing the interior of $\sigma$. Then we have a cofibre sequence
$$\partial\sigma\incl \ind(G)\setminus \inter \sigma \incl \ind(G)\to \susp\,(\partial\sigma)\to\cdots$$
i.e. $\ind(G)$ is the homotopy cofibre of the inclusion $\partial\sigma\incl (\ind(G)\setminus \inter \sigma)$. But this inclusion is null-homotopic since it factors through the space $\ind(M)\setminus \sigma$, homeomorphic to the disk $D^{k-1}$. As a consequence
$$\ind(G)\htpyequiv (\ind(G)\setminus \inter \sigma)\vee\susp\,(\partial\sigma)=(\ind(G)\setminus \inter \sigma)\vee\susp\,S^{k-2}=(\ind(G)\setminus \inter \sigma)\vee S^{k-1}.$$
Now the fact that $\ind(M)\setminus\inter \sigma$ is a contractible subcomplex of $\ind(G)\setminus \inter \sigma$ implies that
$$\ind(G)\setminus \inter \sigma \htpyequiv (\ind(G)\setminus \inter \sigma)/(\ind(M)\setminus \inter\sigma)=\ind(G)/\ind(M)\htpyequiv C(j).$$
Together with $S^{k-1}=\ind(M)$ this ends the proof.

\end{proof}


To analyze our situation further we need to understand homotopy cofibres of maps between independence complexes.

\begin{lemma}
\label{lem:cofibre}
Suppose $i:H\incl G$ is an inclusion of an induced subgraph. Let $C(i)$ denote the graph obtained by adding to $G$ a new vertex adjacent to all the vertices of $V(G)\setminus V(H)$. Then $\ind(C(i))$ is the homotopy cofibre of the induced inclusion $\ind(H)\incl\ind(G)$, i.e. there is a cofibre sequence
$$\ind(H)\incl\ind(G)\incl\ind(C(i))\to\susp\,\ind(H)\to\cdots.$$
\end{lemma}
\begin{proof}
By definition the homotopy cofibre of $\ind(H)\incl\ind(G)$ is obtained by attaching to $\ind(G)$ a cone over the subspace $\ind(H)$. The space $\ind(C(i))$ is obtained in precisely the same way.
\end{proof}

Now we can prove the main result of this section.
\begin{proposition}
\label{prop:many}
Suppose $(M_i,\sigma_i)$, for $i=1,\ldots,k$, is a sequence of induced matchings with dominating transversals in a graph $G$, such that for every $i<j$ we have
\begin{equation}
\label{eq:seq}
\tag{*}
\sigma_j\setminus V(M_i)\neq\emptyset.
\end{equation}
Then there is a homotopy equivalence
$$\ind(G)\htpyequiv\Big(\bigvee_{i=1}^k\ind(M_i)\Big)\vee X$$
for some space $X$.
\end{proposition}
\begin{proof}
Let $i_1:M_1\incl G$ be the inclusion of the first matching. Then, by Lemmas \ref{lem:retract} and \ref{lem:cofibre} we have a splitting 
$$\ind(G)\htpyequiv\ind(M_1)\vee\ind(C(i_1)).$$
If $k=1$ then we are done. Otherwise note that the pairs $(M_j,\sigma_j)$ for $2\leq j\leq k$ define a sequence of induced matchings with dominating transversals in the graph $C(i_1)$. Indeed, to build $C(i_1)$ we did not add any edges within $G$ itself, so the matchings are still induced. Moreover, by (\ref{eq:seq}) every set $\sigma_j$ contains a vertex which is not in $V(M_1)$, hence it is adjacent to the new vertex of $C(i_1)$, which means $\sigma_j$ is a dominating set in $C(i_1)$. The condition (\ref{eq:seq}) still holds, so by induction
$$\ind(C(i_1))\htpyequiv\bigvee_{j=2}^k\ind(M_j)\vee X$$
and that completes the proof.
\end{proof}

\begin{remark}
The last result applies to all the situations of the previous section, once one orders the matchings in the lexicographical order of their defining $\{A,B\}$-words. In particular we get a splitting of $\ind(\hexdim)$ which includes a wedge sum of exponentially many spheres.
\end{remark}

\section{Upper bounds}
\label{sect:upperbounds}

In this section we improve, for the hexagonal dimer grids $\hexdim$, the upper bound of \cite{EngWitten}. We develop a more general result, which is modeled entirely on the technique of \cite{EngWitten} with just two small improvements. Firstly, it avoids discrete Morse theory and relies just on the Betti numbers and homotopy. Secondly, it allows arbitrary graphs in a place where \cite{EngWitten} requires a forest (see Cor.\ref{cor:forest}).

We first need some extra notation. Suppose $K$ is a simplicial complex with a fixed splitting of the vertex set into two disjoint subsets $U$ and $W$ ($V(K)=U\sqcup W$). For every simplex $\sigma\in K[U]$ we denote by $\lk_W\sigma$ the subcomplex of $K[W]$ consisting of those simplices $\tau\in K[W]$ for which $\tau\sqcup\sigma\in K$. By $\st_W\sigma$ we denote the subcomplex of $K$ consisting of those $\tau\in K$ for which $\tau\cap U\subset \sigma$ and $\tau\cap W\in\lk_W\sigma$. Clearly $\st_W\sigma=(\lk_W\sigma)\ast\sigma$.

\begin{example}
If $\sigma=\emptyset$ then we always have $\lk_W\emptyset=\st_W\emptyset=K[W]$. If $U=\{v\}$ then the complexes $\lk_Wv$ and $\st_Wv$ coincide with the usual link and star of $v$ in $K$. The star $\st_W\sigma$ is always contractible when $\sigma\neq\emptyset$.
\end{example}

\begin{lemma}
\label{lem:betti}
Suppose $K$ is a simplicial complex with a vertex partition $U\sqcup W$ as above and such that for every $\sigma\in K[U]$ we have $\redbetti(\lk_W\sigma)\leq B$. Then
$$\redbetti(K)\leq B\cdot|K[U]|.$$
\end{lemma}
\begin{proof}
Denote $D=|K[U]|$.
Fix any ordering $\emptyset=\sigma_0,\sigma_1,\ldots,\sigma_{D-1}$ of the simplices in $K[U]$ such that every simplex is preceded by all its faces. For $0\leq l< D$ define subcomplexes
$$F_l=\bigcup_{i=0}^l \st_W\sigma_i.$$
Then $K[W]=F_0\subset F_1\subset\cdots\subset F_{D-1}=K$ is an increasing, exhaustive filtration of $K$ with quotients:
\begin{equation*}
F_{l}/F_{l-1}=\st_W\sigma_l/(\st_W\sigma_l\cap F_{l-1}) \htpyequiv \susp^{\dim\sigma_l+1}\,\lk_W\sigma_l.
\end{equation*}
Since each of those quotients has total Betti number at most $B$, the $E_1$ page of the homology spectral sequence associated with this filtration has dimension at most $B\cdot D$. The spectral sequence converges to $\redhom_*(K)$, hence the result.
\end{proof}

\begin{lemma}
\label{lem:upperbound}
Suppose $U\subset V(G)$ is a vertex set with the property: For every simplex $\sigma\in \ind(G[U])$ the total Betti number of
$$\ind(G\setminus(U\cup N[\sigma]))$$
is at most $B$. Then
$$\redbetti(\ind(G))\leq B\cdot |\ind(G[U])| \leq B\cdot 2^{|U|}.$$
\end{lemma}
\begin{proof}
Consider $K=\ind(G)$ with vertex set partitioned into $U$ and $W=V(G)\setminus U$. Then for every simplex $\sigma\in \ind(G[U])$ we have precisely
$$\lk_W\sigma=\ind(G\setminus(U\cup N[\sigma])),$$
so Lemma \ref{lem:betti} applies.
\end{proof}

\begin{corollary}[\cite{EngWitten}]
\label{cor:forest}
If $U\subset V(G)$ is a vertex set such that $G\setminus U$ is a forest then $\redbetti(\ind(G))\leq |\ind(G[U])|$.
\end{corollary}
\begin{proof}
If $G\setminus U$ is a forest then so is $G\setminus(U\cup N[\sigma])$ for every $\sigma$. Since by \cite{EH} the independence complex of a forest is either contractible or homotopy equivalent to a sphere (possibly $S^{-1}$), we can apply Lemma \ref{lem:upperbound} with $B=1$.
\end{proof}

\begin{fact}
\label{fact:bettijoin}
For any two topological spaces $X$ and $Y$
$$\redbetti(X\ast Y)=\redbetti(X)\redbetti(Y).$$
\end{fact}
\begin{proof}
This follows from the formula for the reduced homology of the join (eg.\cite[Lemma 2.1]{Milnor}), which for rational coefficients reduces to
$$\redhom_k(X\ast Y)=\bigoplus_{\substack{i,j\geq -1\\ i+j=k-1}} \redhom_i(X)\otimes \redhom_j(Y), \qquad k\geq -1.$$
\end{proof}

\begin{center}
\begin{figure}
\includegraphics[scale=0.8]{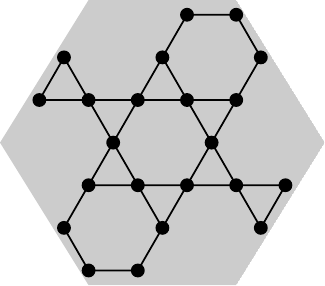}
\label{fig:extremal}
\caption{The induced subgraph $G$ of the $30$-vertex tile with $\redbetti(\ind(G))=14$.}
\end{figure}
\end{center}

\begin{proposition}
\label{prop:hexdimupper}
If $\hexdim$ is any hexagonal dimer lattice which can be tiled as in Fig.\ref{fig:hexdimer} then
$$\redbetti(\hexdim)\leq(14^{1/36}\cdot 2^{1/6})^{v(\hexdim)}.$$
\end{proposition}
\begin{proof}
Recall that $\hexdim$ contains $k=v(\hexdim)/36$ large hexagons. Let $U$ consist of those vertices of $\hexdim$ which are not covered by the tiles. Then $|U|=6k=v(\hexdim)/6$ and the graph $\hexdim\setminus U$ is a disjoint union of $k$ hexagonal tiles with $30$ vertices each. The graph $\hexdim[U]$ has no edges.

Now suppose that $\sigma$ is any subset of $U$. Then the graph $\hexdim\setminus (U\cup N[\sigma])$ is a disjoint union of $k$ graphs, each of which is an induced subgraph of the $30$-vertex tile. More precisely, it is an induced subgraph obtained by removing the neighbourhood of some subset of the vertices of $U$ which surround that tile. There are $2^{12}=4096$ graphs that can arise in this way, with only $217$ isomorphism classes \cite{nauty}, and the homology of their independence complexes can be easily calculated by a computer \cite{Poly}. It turns out that for each of those graphs the total Betti number is at most $14$ (the graph which attains maximum is shown in Fig.\ref{fig:extremal}). Using Fact \ref{fact:bettijoin} we get that the total Betti number of $\hexdim\setminus (U\cup N[\sigma])$ is at most $B=14^k$.

Lemma \ref{lem:upperbound} now gives the conclusion:
$$\redbetti(\ind(\hexdim))\leq 14^k\cdot 2^{|U|}=14^{v(\hexdim)/36}\cdot 2^{v(\hexdim)/6}.$$
\end{proof}

\begin{remark}
If the dimensions of the grid $\hexdim$ do not allow it to be tightly tiled with the large hexagons then Prop.\ref{prop:hexdimupper} still holds asymptotically. This is because one can pack $\hexdim$ with hexagons leaving just a region of size proportional to the perimeter of the grid. The wasted vertices can then be added to the set $U$.
\end{remark}

\begin{remark}
Using exactly the same technique one proves the following counterparts to the lower bounds of \ref{prop:other}:
$$\redbetti(\ind(\tri_3))\leq (2^{3/8})^{v(\tri_3)}, \qquad \redbetti(\ind(\tri_4))\leq (10^{1/45}\cdot 2^{1/5})^{v(\tri_4)}.$$
\end{remark}

\section{Concluding remarks}

It would be most interesting to prove similar lower bounds for the classical triangular and hexagonal lattice, for which superfrustration is also predicted \cite{Eer}. In those cases, however, the present methods do not seem to work, and it is not clear if exponentially many homology classes should be given by embedded spheres or if more complicated constructions are necessary. Recent results in this direction include \cite{HuOthers}, where ``long and thin'' triangular lattices of size $c\times n$ for small constants $2\leq c\leq 7$ are investigated. Also, the constructions of Jonsson \cite{JonGrids} can be adapted to show that for the hexagonal grids of suitable sizes $c\times n$ with fixed $c$ the number of ground states is exponential in $n$.

More generally, it would also be interesting to know what aspect of regularity is responsible for superfrustration in arbitrary lattices.


\end{document}